\documentclass[a4paper,oneside,11pt]{article}%
\usepackage{makeidx}
\usepackage[english]{babel}
\usepackage{amsmath}
\usepackage{amsfonts}
\usepackage{amssymb}
\usepackage{stmaryrd}
\usepackage{graphicx}
\usepackage{mathrsfs}
\usepackage[colorlinks,linkcolor=red,anchorcolor=blue,citecolor=blue,urlcolor=blue]{hyperref}
\usepackage[symbol*,ragged]{footmisc}
\providecommand{\U}[1]{\protect\rule{.1in}{.1in}}
\providecommand{\U}[1]{\protect \rule{.1in}{.1in}}

\newtheorem{theorem}{Theorem}[section]

\newtheorem{lemma}[theorem]{Lemma}

\newtheorem{proposition}[theorem]{Proposition}

\newenvironment{proof}[1][Proof]{\noindent \textbf{#1.} }{\  \rule{0.5em}{0.5em}}
\numberwithin{equation}{section}

\begin{document}

\title{On the Fourth Power Moment of the Error Term for  \\
the Divisor Problem with Congruence Conditions}
%  \author{Min Zhang\footnotemark   \vspace*{-5mm} \\
   %  \small Department of Mathematics, China University of Mining and Technology \vspace*{-5mm} \\
  %  \small  Beijing 100083, P. R. China  }

\author{Jinjiang Li\footnotemark[1]\,\,\,\, \, \& \,\,Min Zhang\footnotemark[2] \vspace*{-4mm} \\
$\textrm{\small Department of Mathematics, China University of Mining and Technology}^{*\,\dag}$
                    \vspace*{-4mm} \\
     \small  Beijing 100083, P. R. China  }

\footnotetext[2]{Corresponding author. \\
    \quad\,\, \textit{ E-mail addresses}: \href{mailto:jinjiang.li.math@gmail.com}{jinjiang.li.math@gmail.com} (J. Li),
     \href{mailto:min.zhang.math@gmail.com}{min.zhang.math@gmail.com} (M. Zhang).    }

\date{}
\maketitle

%\begin{center}
  %\small   Department of Mathematics,
  %  China University of Mining and Technology,
  %  Beijing 100083, P. R. China \\
%\end{center}

{\textbf{Abstract}}: Let $d(n;\ell_1,M_1,\ell_2,M_2)$ denote the number of factorizations $n=n_1n_2$, where each of the factors $n_i\in\mathbb{N}$ belongs to a prescribed congruence class $\ell_i\bmod M_i\,(i=1,2)$. Let
$\Delta(x;\ell_1,M_1,\ell_2,M_2)$ be the error term of the asymptotic formula of $\sum\limits_{n\leqslant x}d(n;\ell_1,M_1,\ell_2,M_2)$. In this
paper, we establish an asymptotic formula of the fourth power moment of $\Delta(M_1M_2x;\ell_1,M_1,\ell_2,M_2)$ and prove that
\begin{equation*}
  \int_1^T\Delta^4(M_1M_2x;\ell_1,M_1,\ell_2,M_2)\mathrm{d}x=\frac{1}{32\pi^4}C_4\Big(\frac{\ell_1}{M_1},\frac{\ell_2}{M_2}\Big)
  T^2+O(T^{2-\vartheta_4+\varepsilon}),
\end{equation*}
with $\vartheta_4=1/8$, which improves the previous value $\theta_4=3/28$ of K. Liu \cite{Liu-Kui}.

{\textbf{Keywords}}: Divisor problem; higher--power moment; asymptotic formula.

{\textbf{MR(2010) Subject Classification}}: 11N37, 11M06.

\section{Introduction and main result}
For fixed integers $M_1,M_2$, the divisor function with congruence conditions is defined by
\begin{equation*}
 d(n;\ell_1,M_1,\ell_2,M_2):=\#\big\{(n_1,n_2)\in\mathbb{N}^2:n_1n_2=n,\,n_i\equiv\ell_i(\bmod M_i),\,i=1,2\big\},
\end{equation*}
where $1\leqslant\ell_1\leqslant M_1$ and $1\leqslant\ell_2\leqslant M_2$. Its generating function is
\begin{equation*}
  \sum_{n=1}^\infty\frac{d(n;\ell_1,M_1,\ell_2,M_2)}{n^s}=\zeta(s,\lambda_1)\zeta(s,\lambda_2)(M_1M_2)^{-s},
\end{equation*}
where $\Re s>1,\lambda_i=\frac{\ell_i}{M_i},i=1,2$, and $\zeta(s,\lambda)$ denotes the Hurwitz zeta--function, which is defined by
\begin{equation*}
  \zeta(s,\alpha):=\sum_{n=0}^{\infty}\frac{1}{(n+\alpha)^s},\qquad s=\sigma+it,\quad \sigma>1,
\end{equation*}
for $0<\alpha\leqslant1$. The ``divisor problem with congruence conditions'' is to study the error term
\begin{align*}
  \Delta(x;\ell_1,M_1,\ell_2,M_2):= & \sum_{n\leqslant x}d(n;\ell_1,M_1,\ell_2,M_2)
                                        \nonumber \\
  & -\sum_{s_0=0,1}Res_{s=s_0}\bigg(\frac{\zeta(s,\lambda_1)\zeta(s,\lambda_2)}{s}\bigg(\frac{x}{M_1M_2}\bigg)^s\bigg).
\end{align*}
It is conjectured that, for every $\varepsilon>0$, there holds
\begin{equation*}
  \Delta(M_1M_2x;\ell_1,M_1,\ell_2,M_2)\ll x^{1/4+\varepsilon}
\end{equation*}
uniformly in $1\leqslant\ell_1\leqslant M_1$ and $1\leqslant\ell_2\leqslant M_2$. The best result up to date is
\begin{equation*}
  \Delta(M_1M_2x;\ell_1,M_1,\ell_2,M_2)\ll x^{131/416}(\log x)^{26947/8320},
\end{equation*}
uniformly again in $1\leqslant\ell_1\leqslant M_1$ and $1\leqslant\ell_2\leqslant M_2$, which follows from Richert's work \cite{Richert}
and Huxley's estimates \cite{Huxley}. In the other direction, Nowak proved in \cite{Nowak-1989} that
\begin{equation*}
  \Delta(Mx;\ell,M,1,1)\ll \Omega_*\Big((x\log x)^{1/4}(\log\log\log x)^{-1/4}\Big)
\end{equation*}
with $*=-$ if $\lambda_1^*<\frac{\ell}{M}<\lambda_2^*$, and $*=+$ else, where $\lambda_1^*=0.03728310\cdots,\lambda_2^*=0.67181895\cdots$ are the zeros of
\begin{equation*}
  -\log(2\sin (\pi\lambda))+\pi\lambda-\frac{\pi}{2}=0
\end{equation*}
in the unit interval. This result holds uniformly in $1\leqslant\ell \leqslant M\leqslant x$, as long as $\frac{\ell}{M}$ is bounded away
from $\lambda_1^*,\lambda_2^*$. Another result follows from Theorem 2 of Nowak \cite{Nowak-1990}, i.e.
\begin{equation*}
  \Delta(x;\ell,M,1,1)=\Omega_\pm\Big((x\log x)^{1/4}(\log\log\log x)^{-1/2}\Big)
\end{equation*}
as $x\to\infty$, where $\ell,M$ are fixed integers satisfying
\begin{equation*}
  0<\frac{\ell}{M}<\frac{1}{6} \qquad \textrm{or} \qquad \frac{1}{2}<\frac{\ell}{M}<\frac{5}{6}.
\end{equation*}
In 1995, M\"{u}ller and Nowak \cite{Muller-Nowak} studied the mean value of $\Delta(M_1M_2x;\ell_1,M_1,\ell_2,M_2)$. They proved that, for a large real parameter $T$, there hold
\begin{equation*}
 \int_1^T\Delta(M_1M_2x;\ell_1,M_1,\ell_2,M_2)\mathrm{d}x\ll T^{3/4}
\end{equation*}
and
\begin{equation*}
 \int_1^T\Delta^2(M_1M_2x;\ell_1,M_1,\ell_2,M_2)\mathrm{d}x=\mathfrak{S}_2\bigg(\frac{\ell_1}{M_1},\frac{\ell_2}{M_2}\bigg)T^{3/2}+o(T^{3/2})
\end{equation*}
(even with a good error term), uniformly in $1\leqslant\ell_1\leqslant M_1,1\leqslant\ell_2\leqslant M_2$. For the third power moment, they proved
that, for any $\varepsilon>0$, there holds
\begin{equation*}
 \int_1^T\Delta^3(M_1M_2x;\ell_1,M_1,\ell_2,M_2)\mathrm{d}x=\mathfrak{S}_3\bigg(\frac{\ell_1}{M_1},\frac{\ell_2}{M_2}\bigg)T^{7/4}+o(T^{47/28+\varepsilon})
\end{equation*}
uniformly in $1\leqslant\ell_1\leqslant M_1,1\leqslant\ell_2\leqslant M_2$. In 2011, Liu \cite{Liu-Kui} consider the higher--power moments of
$\Delta(M_1M_2x;\ell_1,M_1,\ell_2,M_2)$. In his paper, Liu refer to a unified approach, which is originated from Zhai \cite{Zhai-2}, to show
the asymptotic formula
\begin{equation}\label{unified-Delta-asymptotic}
  \int_1^T\Delta^k(M_1M_2x;\ell_1,M_1,\ell_2,M_2)\mathrm{d}x=\mathcal{C}_k
  T^{1+k/4}+O(T^{1+k/4-\vartheta_k+\varepsilon}).
\end{equation}
holds for $3\leqslant k\leqslant9$, where $\mathcal{C}_k$ and $0<\vartheta_k<1$ are explicit constants. For the case $k=4$, Liu \cite{Liu-Kui} also
gives a separate conclusion of the value $\vartheta_4=3/28$, which is better than that in the unified asymptotic formula
(\ref{unified-Delta-asymptotic}).

The aim of this paper is to improve the value of $\vartheta_4=3/28$, which is achieved
by Liu \cite{Liu-Kui}. The main result is the following theorem.

\begin{theorem}\label{theorem-1}
 We have
\begin{equation*}
   \int_1^T\Delta^4(M_1M_2x;\ell_1,M_1,\ell_2,M_2)\mathrm{d}x=\frac{1}{32\pi^4}C_4\bigg(\frac{\ell_1}{M_1},\frac{\ell_2}{M_2}\bigg)
  T^2+O(T^{2-\vartheta_4+\varepsilon}),
\end{equation*}
with $\vartheta_4=1/8$, where $C_4(\frac{\ell_1}{M_1},\frac{\ell_2}{M_2})$ is defined by (\ref{C_4(l_1/M_1,l_2/M_2)-def}).
\end{theorem}

We will use the method, which is developed by Zhai \cite{Zhai-3}, to establish Theorem \ref{theorem-1}, but with Kong's Lemma (See Lemma \ref{kong-lamma}) instead of Lemma 2 in Zhai \cite{Zhai-3} by more careful analysis. This method can be applied to a series of fourth--power moment estimates , and
obtained  better estimates.

\section{Notation}
Throughout this paper, $\|x\|$ denotes the distance from $x$ to the nearest integer, i.e. $\|x\|=\min_{n\in\mathbb{Z}}|x-n|$. $[x]$ denotes the integer
part of $x$; $\psi(t)=t-[t]-\frac{1}{2}$; $e(t)=e^{2\pi it}$.
$n\sim N$ means $N<n\leqslant2N$; $n\asymp N$ means $C_1N\leqslant n\leqslant C_2N$ with positive constants $C_1,C_2$ satisfying $C_1<C_2$.
$f(x)=\Omega_+(g(x))$ means that there exists a suitable constant $C>0$ such that $f(x)>Cg(x)$ holds for a sequence $x=x_n$ such that
$\lim_{n\to\infty}x_n=\infty$; $f(x)=\Omega_-(g(x))$ means that there exists a suitable constant $C>0$ such that $f(x)<-Cg(x)$ holds for a sequence
$x=x_n$ such that $\lim_{n\to\infty}x_n=\infty$; $f(x)=\Omega(g(x))$ means that $|f(x)|=\Omega_\pm(g(x))$.
$\varepsilon$ always denotes an arbitrary small positive constant, which may not be the same at different occurrences. $d(n)$ denotes the Dirichlet
divisor  function, and we shall use the well--know estimate $d(n)\ll n^\varepsilon$. Suppose $f:\mathbb{N}\to\mathbb{R}$ is any arithmetic function
satisfying $f(n)\ll n^\varepsilon$. Define
\begin{equation*}
   \sideset{}{'}\sum_{\alpha\leqslant n\leqslant\beta}f(n)=\left\{
    \begin{array}{ll}
       \displaystyle\sum_{\alpha<n<\beta}f(n), & \textrm{if $\alpha\not\in\mathbb{Z},\beta\not\in\mathbb{Z}$}, \\
       \displaystyle\frac{1}{2}f(\alpha)+\displaystyle\sum_{\alpha<n<\beta}f(n), & \textrm{if $\alpha\in\mathbb{Z},\beta\not\in\mathbb{Z}$}, \\
       \displaystyle\displaystyle\sum_{\alpha<n<\beta}f(n)+\frac{1}{2}f(\beta), & \textrm{if $\alpha\not\in\mathbb{Z},\beta\in\mathbb{Z}$}, \\
       \displaystyle\frac{1}{2}f(\alpha)+\displaystyle\sum_{\alpha<n<\beta}f(n)+\frac{1}{2}f(\beta), & \textrm{if $\alpha\in\mathbb{Z},\beta\in\mathbb{Z}$},
    \end{array}
   \right.
\end{equation*}
and, for $k\geqslant2$, define
\begin{equation}\label{s_{k;v}-def}
  s_{k;v}:=s_{k;v}(f)=\sum_{\substack{n_1,\cdots,n_v,n_{v+1},\cdots,n_k\in\mathbb{N}^*\\ \sqrt{n_1}+\cdots+\sqrt{n_v}=\sqrt{n_{v+1}}+\cdots+\sqrt{n_k}}}
  \frac{f(n_1)f(n_2)\cdots f(n_k)}{(n_1n_2\cdots n_k)^{3/4}}, \quad 1\leqslant v<k.
\end{equation}
We shall use $s_{k;v}$ or $s_{k;v}(f)$ to denote both of the series (\ref{s_{k;v}-def}) and its value. Suppose $y>1$ is large parameter, and we define
\begin{equation*}
  s_{k;v}(y):=s_{k;v}(f;y)=\sum_{\substack{n_1,\cdots,n_v,n_{v+1},\cdots,n_k\leqslant y\\ \sqrt{n_1}+\cdots+\sqrt{n_v}=\sqrt{n_{v+1}}+\cdots+\sqrt{n_k}}}
  \frac{f(n_1)f(n_2)\cdots f(n_k)}{(n_1n_2\cdots n_k)^{3/4}}, \quad 1\leqslant v<k.
\end{equation*}

\section{Preliminary Lemmas}

\begin{lemma}\label{Heath-Brown-psi(u)}
   Let $H\geqslant2$ be any real number. Then we have
\begin{equation*}
   \psi(u)=-\sum_{1\leqslant|h|\leqslant H}\frac{e(hu)}{2\pi ih}+O\bigg(\min\bigg(1,\frac{1}{H\|u\|}\bigg)\bigg).
\end{equation*}
\end{lemma}
\begin{proof}
  See pp. 245 of Heath--Brown \cite{Heath-Brown-1}.
\end{proof}

\begin{lemma}\label{Min-lemma}
Suppose $A_1,\cdots,A_5$ are absolute positive constants. Let $f(x)$ and $g(x)$ are algebraic functions on $[a,b]$, which satisfy
\begin{equation*}
  \frac{A_1}{R}\leqslant|f''(x)|\leqslant\frac{A_2}{R}, \quad |f'''(x)|\leqslant\frac{A_3}{RU},\quad U\geqslant1,
\end{equation*}
\begin{equation*}
|g(x)|\leqslant A_4G, \quad |g'(x)|\leqslant\frac{A_5G}{U_1},\quad U_1\geqslant1.
\end{equation*}
Suppose $[\alpha,\beta]$ is the image of $[a,b]$ under the mapping $y=f'(x)$, then
\begin{align*}
       \sum_{a<n\leqslant b}g(n)e\big(f(n)\big)
  =& e^{\frac{\pi i}{4}}\sum_{\alpha\leqslant \nu\leqslant \beta}b_\nu\frac{g(n_\nu)}{\sqrt{f''(n_\nu)}}e\big(f(n_\nu)-\nu n_\nu\big)   \\
  & +O\big(G\log(\beta-\alpha+2)+G(b-a+R)(U^{-1}+U_1^{-1})\big)  \\
  & +O\bigg(G\min\bigg(\sqrt{R},\,\,\max\bigg(\frac{1}{\langle\alpha\rangle},\,\frac{1}{\langle\beta\rangle}\bigg)\bigg)\bigg),
\end{align*}
where $n\nu$ is the solution of $f'(n)=\nu$,
\begin{align*}
  \langle t\rangle= &\left\{
    \begin{array}{cl}
       \displaystyle\|t\|, & \textrm{if $t\not\in\mathbb{Z}$},  \\
       \displaystyle\beta-\alpha, & \textrm{if $t\in\mathbb{Z}$},
    \end{array}
        \right.    \\
  b_\nu=& \left\{
   \begin{array}{cl}
      \displaystyle 1, & \textrm{if $\alpha<\nu<\beta$ or $\alpha\not\in\mathbb{Z},\beta\not\in\mathbb{Z}$},  \\
       \displaystyle\frac{1}{2}, & \textrm{if $u=\alpha\in\mathbb{Z}$ or $u=\beta\in\mathbb{Z}$ },
    \end{array}
        \right.      \\
\sqrt{f''}=& \left\{
   \begin{array}{cl}
       \sqrt{f''}, & \textrm{if $f''>0$},  \\
       i\sqrt{|f''|}, & \textrm{if $f''<0$ }.
    \end{array}
  \right.
\end{align*}

\end{lemma}
\begin{proof}
  See Theorem 2.2 in Chapter 2 in Min \cite{Min-si-he}, also see Theorem 1 in Chapter III in Karatsuba and Voronin \cite{Karatsuba-Voronin}.
\end{proof}

\begin{lemma}\label{Tsang-lemma-1}
If $g(x)$ and $h(x)$ are continuous real--valued functions of $x$ and $g(x)$ is monotonic, then
\begin{equation*}
 \int_a^bg(x)h(x)\mathrm{d}x\ll \bigg(\max_{a\leqslant x\leqslant b}|g(x)|\bigg)\bigg
  (\max_{a\leqslant u<v\leqslant b}\bigg|\int_u^vh(x)\mathrm{d}x\bigg|\bigg).
\end{equation*}
\end{lemma}
\begin{proof}
  See Lemma 1 of Tsang \cite{Tsang}.
\end{proof}

\begin{lemma}\label{yijie}
 Suppose $A,B\in\mathbb{R},\,A\not=0$. Then for any $\alpha\in\mathbb{R}$, there holds
 \begin{equation*}
     \int_T^{2T} t^\alpha \cos\big(A\sqrt{t}+B\big)\mathrm{d}t\ll T^{1/2+\alpha}|A|^{-1}.
 \end{equation*}
\end{lemma}
\begin{proof}
  It follows from Lemma \ref{Tsang-lemma-1} easily.
\end{proof}

\begin{lemma}\label{s_k;v(f)=s_k;v(f;y)+error}
  Let f(n) be an arithmetic function. Then we have
\begin{equation*}
   |s_{k;v}(f)-s_{k;v}(f;y)|\ll y^{-1/2+\varepsilon}, \qquad 1\leqslant v<k.
\end{equation*}
\end{lemma}
\begin{proof}
  See Lemma 3.1 of Zhai \cite{Zhai-2}.
\end{proof}

\begin{lemma}\label{kong-lamma}
   If $n,m,k,\ell\in\mathbb{N}$ such that $\sqrt{n}+\sqrt{m}\pm\sqrt{k}-\sqrt{\ell}\not=0$, then there hold
   \begin{equation*}
      |\sqrt{n}+\sqrt{m}\pm\sqrt{k}-\sqrt{\ell}|\gg(nmk\ell)^{-1/2}\max(n,m,k,\ell)^{-3/2},
   \end{equation*}
    respectively.
\end{lemma}
\begin{proof}
   See Kong~\cite{Kong}, Lemma~3.2.1.
\end{proof}

\begin{lemma}\label{Zhai-lemma-5}
   Suppose $1\leqslant N\leqslant M,\,1\leqslant K\leqslant L,\,N\leqslant K,\,M\asymp L,\,0<\Delta\ll L^{1/2}$. Let $\mathscr{A}_1(N,M,K,L;\Delta)$
denote the number of solutions of the following inequality
\begin{equation*}
    0<|\sqrt{n}+\sqrt{m}-\sqrt{k}-\sqrt{\ell}|<\Delta
\end{equation*}
with $n\sim N,\,m\sim M,\,k\sim K,\,\ell\sim L$. Then we have
\begin{equation*}
   \mathscr{A}_1(N,M,K,L;\Delta)\ll \Delta L^{1/2}NMK+NKL^{1/2+\varepsilon}.
\end{equation*}
Especially, if $\Delta L^{1/2}\gg1$, then
\begin{equation*}
  \mathscr{A}_1(N,M,K,L;\Delta)\ll \Delta L^{1/2}NMK.
\end{equation*}
\end{lemma}
\begin{proof}
   See Zhai~\cite{Zhai-3}, Lemma 5.
\end{proof}

\begin{lemma}\label{Zhai-lemma-6}
  Suppose $1\leqslant N\leqslant M\leqslant K\asymp L,\,0<\Delta\ll L^{1/2}$, Let $\mathscr{A}_2(N,M,K,L;\Delta)$ denote the number
  of solutions of the inequality
\begin{equation*}
    0<|\sqrt{n}+\sqrt{m}+\sqrt{k}-\sqrt{\ell}|<\Delta
\end{equation*}
with $n\sim N,\,m\sim M,\,k\sim K,\,\ell\sim L$. Then we have
\begin{equation*}
   \mathscr{A}_2(N,M,K,L;\Delta)\ll \Delta L^{1/2}NMK+NML^{1/2+\varepsilon}.
\end{equation*}
Especially, if $\Delta L^{1/2}\gg1$, then
\begin{equation*}
  \mathscr{A}_1(N,M,K,L;\Delta)\ll \Delta L^{1/2}NMK.
\end{equation*}
\end{lemma}
\begin{proof}
   See Zhai \cite{Zhai-3}, Lemma 6.
\end{proof}

\begin{lemma}\label{Zhai-lemma-3}
   Suppose $N_j\geqslant2\,(j=1,2,3,4),\,\Delta>0$. Let $\mathscr{A}_\pm(N_1,N_2,N_3,N_4;\Delta)$ denote the number of
     solutions of the following inequality
\begin{equation*}
   0<|\sqrt{n_1}+\sqrt{n_2}\pm\sqrt{n_3}-\sqrt{n_4}|<\Delta
\end{equation*}
with $n_j\sim N_j\,(j=1,2,3,4),\,n_j\in\mathbb{N}^*$. Then we have
\begin{equation*}
  \mathscr{A}_\pm(N_1,N_2,N_3,N_4;\Delta)\ll\prod_{j=1}^4\big(\Delta^{1/4}N_j^{7/8}+N_j^{1/2}\big)N_j^{\varepsilon}.
\end{equation*}
\end{lemma}
\begin{proof}
   See Zhai~\cite{Zhai-3}, Lemma 3.
\end{proof}

\section{Analogue of Vorono\"{i}'s Formula}
For the Dirichlet divisor problem, there exists the following truncated Vorono\"{i}'s formula, i.e.
\begin{equation*}
  \Delta(x)=\frac{x^{1/4}}{\sqrt{2}\pi}\sum_{n\leqslant N}\frac{d(n)}{n^{3/4}}\cos\big(4\pi\sqrt{nx}-\pi/4\big)+O\big(x^{1/2+\varepsilon}N^{-1/2}\big),
\end{equation*}
for $1\leqslant N\ll x$. This formula plays a significant role in the study of the higher--power moments of $\Delta(x)$.
But for $\Delta(M_1M_2x;\ell_1,M_1,\ell_2,M_2)$, there is not such a convenient formula at hand. However, we can use the finite expression of $\psi(x)$
and van der Corput's B--process to establish an analogue of Vorono\"{i}'s formula.

Suppose $T/2\leqslant x\leqslant T$. According to Dirichlet's hyperbolic summation method, we write
\begin{equation*}
  \sum_{n\leqslant M_1M_2x}d(n;\ell_1,M_1,\ell_2,M_2)=\sum_{\substack{n_1n_2\leqslant M_1M_2x\\ n_1\equiv \ell_1(\!\bmod M_1)\\n_2\equiv\ell_2(\!\bmod M_2)}}1
  =\Sigma_1+\Sigma_2-\Sigma_3,
\end{equation*}
where
\begin{equation*}
   \Sigma_1=\sum_{\substack{n_1\leqslant\sqrt{M_1M_2x}\\ n_1\equiv \ell_1(\!\bmod M_1)}}
           \sum_{\substack{n_2\leqslant\frac{M_1M_2x}{n_1}\\n_2\equiv \ell_2(\!\bmod M_2) }}1, \qquad
       \Sigma_2=\sum_{\substack{n_2\leqslant\sqrt{M_1M_2x}\\ n_2\equiv \ell_2(\!\bmod M_2)}}
           \sum_{\substack{n_1\leqslant\frac{M_1M_2x}{n_2}\\n_1\equiv \ell_1(\!\bmod M_1) }}1,
\end{equation*}
\begin{equation*}
 \Sigma_3=\sum_{\substack{n_1\leqslant\sqrt{M_1M_2x}\\ n_1\equiv \ell_1(\!\bmod M_1)}} 1
              \sum_{\substack{n_2\leqslant\sqrt{M_1M_2x}\\ n_2\equiv \ell_2(\!\bmod M_2)}} 1.
\end{equation*}
Computing the above three sums directly, it is easy to derive that
\begin{equation}\label{Delta=F_12+F_21}
 \Delta(M_1M_2x;\ell_1,M_1,\ell_2,M_2)=F_{12}(x)+F_{21}(x)+O(1),
\end{equation}
where
\begin{equation*}
 F_{12}(x):=-\sum_{\substack{n_1\leqslant\sqrt{M_1M_2x}\\ n_1\equiv\ell_1(\!\bmod M_1)}}\psi\bigg(\frac{M_1x}{n_1}-\frac{\ell_2}{M_2}\bigg),
\end{equation*}
\begin{equation*}
 F_{21}(x):=-\sum_{\substack{n_1\leqslant\sqrt{M_1M_2x}\\ n_2\equiv\ell_2(\!\bmod M_2)}}\psi\bigg(\frac{M_2x}{n_2}-\frac{\ell_1}{M_1}\bigg).
\end{equation*}
First, we consider $F_{12}(x)$. Suppose that $H\geqslant2$ is a parameter which is to be determined later. By Lemma \ref{Heath-Brown-psi(u)}, we get
\begin{equation}\label{F_12=R_12+G_12}
 F_{12}(x)=R_{12}(x;H)+G_{12}(x;H),
\end{equation}
where
\begin{equation*}
 R_{12}(x;H):=\frac{1}{2\pi i}\sum_{1\leqslant |h|\leqslant H}\frac{1}{h}\sum_{\substack{n_1\leqslant\sqrt{M_1M_2x}\\ n_1\equiv\ell_1(\!\bmod M_1)}}
 e\bigg(\frac{hM_1x}{n_1}-\frac{h\ell_2}{M_2}\bigg),
\end{equation*}
\begin{equation}\label{G_12(x;H)-def}
 G_{12}(x;H):=O\left(\sum_{\substack{n_1\leqslant\sqrt{M_1M_2x}\\ n_1\equiv\ell_1(\!\bmod M_1)}}
    \min\Bigg(1,\frac{1}{H\big\|\frac{M_1x}{n_1}-\frac{\ell_2}{M_2}\big\|}\Bigg)\right).
\end{equation}
By a splitting argument, we have
\begin{equation*}
 R_{12}(x;H)=\frac{1}{2\pi i}\sum_{1\leqslant |h|\leqslant H}\frac{1}{h}\sum_{j=0}^J
 \sum_{\substack{\sqrt{\frac{M_1M_2x}{2^{j+1}}}<n_1\leqslant\sqrt{\frac{M_1M_2x}{2^j}} \\ n_1\equiv\ell_1(\!\bmod M_1)}}
 e\bigg(\frac{hM_1x}{n_1}-\frac{h\ell_2}{M_2}\bigg)+O(\mathscr{L}^2),
\end{equation*}
where
\begin{equation*}
 J=\bigg[\frac{\mathscr{L}-2\log\mathscr{L}}{\log2}\bigg].
\end{equation*}
It is easy to see that
\begin{equation}\label{R_12-Sigma_12}
 R_{12}(x;H)=-\frac{\Sigma_{12}}{2\pi i}+\frac{\overline{\Sigma_{12}}}{2\pi i}+O(\mathscr{L}^2),
\end{equation}
where
\begin{equation}\label{Sigma_{12}}
\Sigma_{12}=\sum_{1\leqslant h\leqslant H}\frac{1}{h}\sum_{j=0}^J\mathcal{S}(x,h,j),
\end{equation}
\begin{equation}\label{S(x,h,j)-def}
\mathcal{S}(x,h,j):= \sum_{\substack{\sqrt{\frac{M_1M_2x}{2^{j+1}}}<n_1\leqslant\sqrt{\frac{M_1M_2x}{2^j}} \\ n_1\equiv\ell_1(\!\bmod M_1)}}
 e\bigg(-\frac{hM_1x}{n_1}+\frac{h\ell_2}{M_2}\bigg).
\end{equation}
Applying Lemma \ref{Min-lemma} to (\ref{S(x,h,j)-def}), we obtain
\begin{align}\label{S(x,h,j)-transfer}
  \mathcal{S}(x,h,j) = & \sum_{\substack{k\\ \sqrt{\frac{M_1M_2x}{2^{j+1}}}<kM_1+\ell_1\leqslant\sqrt{\frac{M_1M_2x}{2^j}}  }}
                            e\bigg(-\frac{hM_1x}{kM_1+\ell_1}+\frac{h\ell_2}{M_2}\bigg)   \nonumber \\
     = & \frac{x^{1/4}}{\sqrt{2}}\sideset{}{'}\sum_{\frac{2^jhM_1}{M_2}\leqslant r\leqslant \frac{2^{j+1}hM_1}{M_2}}\frac{h^{1/4}}{r^{3/4}}
          e\bigg(-2\sqrt{hrx}+\frac{h\ell_2}{M_2}+\frac{r\ell_1}{M_1}-\frac{1}{8}\bigg)+O(\mathscr{L}).
\end{align}
Putting (\ref{S(x,h,j)-transfer}) into (\ref{Sigma_{12}}), we get
\begin{align}\label{Sigma_(12)-transfer}
  \Sigma_{12} = & \frac{x^{1/4}}{\sqrt{2}}\sum_{1\leqslant h\leqslant H}\sum_{j=0}^J
             \sideset{}{'}\sum_{\frac{2^jhM_1}{M_2}\leqslant r\leqslant \frac{2^{j+1}hM_1}{M_2}}\frac{1}{(hr)^{3/4}}
              e\bigg(-2\sqrt{hrx}+\frac{h\ell_2}{M_2}+\frac{r\ell_1}{M_1}-\frac{1}{8}\bigg)+O(\mathscr{L}^3)
                   \nonumber  \\
   = & \frac{x^{1/4}}{\sqrt{2}}\sum_{1\leqslant h\leqslant H}
              \sideset{}{'}\sum_{\frac{hM_1}{M_2}\leqslant r\leqslant\frac{2^{J+1}hM_1}{M_2}}\frac{1}{(hr)^{3/4}}
              e\bigg(-2\sqrt{hrx}+\frac{h\ell_2}{M_2}+\frac{r\ell_1}{M_1}-\frac{1}{8}\bigg)+O(\mathscr{L}^3).
\end{align}
Define
\begin{equation*}
  \tau_{12}(n,x;H):=\sideset{}{'}\sum_{\substack{n=hr\\ 1\leqslant h\leqslant H\\ \frac{hM_1}{M_2}\leqslant r\leqslant\frac{2^{J+1}hM_1}{M_2}}}
  \cos\bigg(4\pi\sqrt{nx}-2\pi\bigg(\frac{h\ell_2}{M_2}+\frac{r\ell_1}{M_1}+\frac{1}{8}\bigg)\bigg)
\end{equation*}
and
\begin{equation*}
  R_{12}'(x;H):=\frac{x^{1/4}}{\sqrt{2}\pi}\sum_{n\leqslant \frac{2^{J+1}H^2M_1}{M_2}}\frac{\tau_{12}(n,x;H)}{n^{3/4}}.
\end{equation*}
Inserting (\ref{Sigma_(12)-transfer}) into (\ref{R_12-Sigma_12}), we obtain
\begin{align}\label{R_12=R12'+L^3}
    &  R_{12}(x;H)
               \nonumber  \\
    = & \frac{x^{1/4}}{2\sqrt{2}\pi i}\sum_{1\leqslant h\leqslant H}\sideset{}{'}\sum_{\frac{hM_1}{M_2}\leqslant r\leqslant\frac{2^{J+1}hM_1}{M_2}}
                \frac{1}{(hr)^{3/4}}
                        \nonumber  \\
   & \times \bigg(e\bigg(2\sqrt{hrx}-\frac{h\ell_2}{M_2}-\frac{r\ell_1}{M_1}+\frac{1}{8}\bigg)-
                  e\bigg(-2\sqrt{hrx}+\frac{h\ell_2}{M_2}+\frac{r\ell_1}{M_1}-\frac{1}{8}\bigg)\bigg)
               +O(\mathscr{L}^3)
                        \nonumber  \\
    = & \frac{x^{1/4}}{\sqrt{2}\pi }\sum_{1\leqslant h\leqslant H}\sideset{}{'}\sum_{\frac{hM_1}{M_2}\leqslant r\leqslant\frac{2^{J+1}hM_1}{M_2}}
                \frac{1}{(hr)^{3/4}}
       \sin\bigg(4\pi\sqrt{hrx}-2\pi\bigg(\frac{h\ell_2}{M_1}+\frac{r\ell_1}{M_1}-\frac{1}{8}\bigg)\bigg)+O(\mathscr{L}^3)
                         \nonumber  \\
    = & \frac{x^{1/4}}{\sqrt{2}\pi }\sum_{1\leqslant h\leqslant H}\sideset{}{'}\sum_{\frac{hM_1}{M_2}\leqslant r\leqslant\frac{2^{J+1}hM_1}{M_2}}
                \frac{1}{(hr)^{3/4}}
       \cos\bigg(4\pi\sqrt{hrx}-2\pi\bigg(\frac{h\ell_2}{M_1}+\frac{r\ell_1}{M_1}+\frac{1}{8}\bigg)\bigg)+O(\mathscr{L}^3)
                        \nonumber  \\
    = & R'_{12}(x;H)+O(\mathscr{L}^3).
\end{align}
Define
\begin{equation*}
 \tau_{12}(n,x):=\sideset{}{'}\sum_{\substack{n=hr\\ \frac{hM_1}{M_2}\leqslant r}}
                 \cos\bigg(4\pi\sqrt{nx}-2\pi\bigg(\frac{h\ell_2}{M_2}+\frac{r\ell_1}{M_1}+\frac{1}{8}\bigg)\bigg).
\end{equation*}
It is easy to check that if $n\leqslant\min(H^2,T)\mathscr{L}^{-3}$, then $h\leqslant H$ and $r\leqslant\frac{2^{J+1}hM_1}{M_2}$. Therefore, we have
\begin{equation*}
 \tau_{12}(n,x;H)=\tau_{12}(n,x),\qquad n\leqslant \min(H^2,T)\mathscr{L}^{-3}.
\end{equation*}
Suppose $T^\varepsilon<y\leqslant\min(H^2,T)\mathscr{L}^{-3}$ is a parameter which is to be determined later. Define
\begin{equation}\label{R_(12)(x;y)-def}
  R_{12}(x;y):=\frac{x^{1/4}}{\sqrt{2}\pi}\sum_{n\leqslant y}\frac{\tau_{12}(n,x)}{n^{3/4}}
\end{equation}
and
\begin{equation}\label{R_(12)(x;y,H)-def}
  R^*_{12}(x;y,H):=\frac{x^{1/4}}{\sqrt{2}\pi}\sum_{y<n\leqslant \frac{2^{J+1}H^2M_1}{M_2}}\frac{\tau_{12}(n,x;H)}{n^{3/4}}.
\end{equation}
Then from (\ref{F_12=R_12+G_12}), (\ref{R_12=R12'+L^3}), (\ref{R_(12)(x;y)-def}) and (\ref{R_(12)(x;y,H)-def}), we get
\begin{equation}\label{F_12-fenjie}
  F_{12}(x)=R_{12}(x;y)+R_{12}^*(x;y,H)+G_{12}(x;H)+O(\mathscr{L}^3).
\end{equation}
By symmetry, we also have
\begin{equation}\label{F_21-fenjie}
  F_{21}(x)=R_{21}(x;y)+R_{21}^*(x;y,H)+G_{21}(x;H)+O(\mathscr{L}^3),
\end{equation}
where
\begin{equation*}
  R_{21}(x;y)=\frac{x^{1/4}}{\sqrt{2}\pi}\sum_{n\leqslant y}\frac{\tau_{21}(n,x)}{n^{3/4}},
\end{equation*}
\begin{equation*}
 \tau_{21}(n,x)=\sideset{}{'}\sum_{\substack{n=hr\\ \frac{hM_2}{M_1}\leqslant r}}
                 \cos\bigg(4\pi\sqrt{nx}-2\pi\bigg(\frac{h\ell_1}{M_1}+\frac{r\ell_2}{M_2}+\frac{1}{8}\bigg)\bigg),
\end{equation*}
\begin{equation*}
  R^*_{21}(x;y,H)=\frac{x^{1/4}}{\sqrt{2}\pi}\sum_{y<n\leqslant \frac{2^{J+1}H^2M_2}{M_1}}\frac{\tau_{21}(n,x;H)}{n^{3/4}},
\end{equation*}
\begin{equation*}
  \tau_{21}(n,x;H)=\sideset{}{'}\sum_{\substack{n=hr\\ 1\leqslant h\leqslant H\\ \frac{hM_2}{M_1}\leqslant r\leqslant\frac{2^{J+1}hM_2}{M_1}}}
  \cos\bigg(4\pi\sqrt{nx}-2\pi\bigg(\frac{h\ell_1}{M_1}+\frac{r\ell_2}{M_2}+\frac{1}{8}\bigg)\bigg),
\end{equation*}
\begin{equation*}
 G_{21}(x;H)=O\left(\sum_{\substack{n_2\leqslant\sqrt{M_1M_2x}\\ n_2\equiv\ell_2(\!\bmod M_2)}}
    \min\Bigg(1,\frac{1}{H\big\|\frac{M_2x}{n_2}-\frac{\ell_1}{M_1}\big\|}\Bigg)\right).
\end{equation*}
Define
\begin{equation*}
 \tau(n,x):=\sum_{n=hr}\cos\bigg(4\pi\sqrt{nx}-2\pi\bigg(\frac{h\ell_2}{M_2}+\frac{r\ell_1}{M_1}+\frac{1}{8}\bigg)\bigg)
\end{equation*}
and
\begin{equation*}
  R_0(x;y):=\frac{x^{1/4}}{\sqrt{2}\pi}\sum_{n\leqslant y}\frac{\tau(n,x)}{n^{3/4}}.
\end{equation*}
From the definition of $\tau_{12}(n,x)$ and $\tau_{21}(n,x)$, we derive that
\begin{align*}
  \tau_{12}(n,x)+\tau_{21}(n,x) = & \sideset{}{'}\sum_{\substack{n=hr\\ \frac{hM_1}{M_2}\leqslant r}}
                 \cos\bigg(4\pi\sqrt{nx}-2\pi\bigg(\frac{h\ell_2}{M_2}+\frac{r\ell_1}{M_1}+\frac{1}{8}\bigg)\bigg)
                                              \nonumber \\
                               & +\sideset{}{'}\sum_{\substack{n=hr\\ \frac{hM_2}{M_1}\leqslant r}}
                                  \cos\bigg(4\pi\sqrt{nx}-2\pi\bigg(\frac{h\ell_1}{M_1}+\frac{r\ell_2}{M_2}+\frac{1}{8}\bigg)\bigg)
                                              \nonumber \\
                            =  & \sideset{}{'}\sum_{\substack{n=hr\\ \frac{hM_1}{M_2}\leqslant r}}
                 \cos\bigg(4\pi\sqrt{nx}-2\pi\bigg(\frac{h\ell_2}{M_2}+\frac{r\ell_1}{M_1}+\frac{1}{8}\bigg)\bigg)
                                              \nonumber \\
\end{align*}
\begin{align*}
                               & +\sideset{}{'}\sum_{\substack{n=hr\\ r\leqslant\frac{hM_1}{M_2}}}
                                  \cos\bigg(4\pi\sqrt{nx}-2\pi\bigg(\frac{h\ell_2}{M_2}+\frac{r\ell_1}{M_1}+\frac{1}{8}\bigg)\bigg)
                                              \nonumber \\
                            =  & \sum_{n=hr}\cos\bigg(4\pi\sqrt{nx}-2\pi\bigg(\frac{h\ell_2}{M_2}+\frac{r\ell_1}{M_1}+\frac{1}{8}\bigg)\bigg)
                                 =\tau(n,x),
\end{align*}
which implies
\begin{equation}\label{R_0=R_12+R_21}
  R_0(x;y)=R_{12}(x;y)+R_{21}(x;y).
\end{equation}
Combining (\ref{Delta=F_12+F_21}), (\ref{F_12-fenjie})--(\ref{R_0=R_12+R_21}), we obtain
\begin{align}\label{Delta=R_0+R12+R_21+G_12+G_21}
 \Delta(M_1M_2x;\ell_1,M_1,\ell_2,M_2)= & R_0(x;y)+R_{12}^*(x;y,H)+R_{21}^*(x;y,H)
                                               \nonumber \\
                                        &    +G_{12}(x;H)+G_{21}(x;H)+O(\mathscr{L}^3).
\end{align}
For simplicity, we denote $\mathscr{R}_0=R_0(x;y),\,\mathscr{R}_1=R_{12}^*(x;y,H)$, $\mathscr{R}_2=R_{21}^*(x;y,H)$,
$\mathscr{R}=\mathscr{R}_0+\mathscr{R}_1+\mathscr{R}_2$ and $\mathscr{G}=G_{12}(x;H)+G_{21}(x;H)$. Then (\ref{Delta=R_0+R12+R_21+G_12+G_21}) can be
written as
\begin {equation*}
 \Delta(M_1M_2x;\ell_1,M_1,\ell_2,M_2)= \mathscr{R}+\mathscr{G}+O(\mathscr{L}^3).
\end{equation*}

\section{Proof of Theorem \ref{theorem-1}}

In this section, we shall prove Theorem \ref{theorem-1}. Take $H=T^8$ and $y=T^{3/4}$. By the elementary formula
\begin{equation}\label{elementary-formula }
  (a+b)^4=a^4+O(|a|^3|b|+b^4),
\end{equation}
we have
\begin {equation}\label{Delta(M_1M_2x)-fenjie}
 \Delta^4(M_1M_2x;\ell_1,M_1,\ell_2,M_2)= \mathscr{R}^4+O\big(|\mathscr{R}|^3|\mathscr{G}|+|\mathscr{R}|^3\mathscr{L}^3+\mathscr{G}^4+\mathscr{L}^{12}\big).
\end{equation}

By a splitting argument, it is sufficient to prove the result in the interval $[T/2,T]$. We will divide the process of the proof of Theorem \ref{theorem-1}
into two parts.

\begin{proposition}\label{proposition-1}
   For $T\geqslant10,\,y=T^{3/4}$, we have
\begin{equation*}
\int_{\frac{T}{2}}^T\mathscr{R}^4\mathrm{d}x=\frac{1}{32\pi^4}C_4\bigg(\frac{\ell_1}{M_1},\frac{\ell_2}{M_2}\bigg)\int_{\frac{T}{2}}^Tx\mathrm{d}x
+O(T^{2-1/8+\varepsilon}).
\end{equation*}
\end{proposition}
\begin{proof}
  From (\ref{elementary-formula }), we get
\begin{equation*}
   \mathscr{R}^4=\big(\mathscr{R}_0+(\mathscr{R}_1+\mathscr{R}_2)\big)^4=\mathscr{R}_0^4+O\big(|\mathscr{R}_0|^3|\mathscr{R}_1
                                          +\mathscr{R}_2|+|\mathscr{R}_1+\mathscr{R}_2|^4\big)
\end{equation*}
Therefore, we have
\begin{equation}\label{R_4-mean-explicit}
   \int_{\frac{T}{2}}^T\mathscr{R}^4\mathrm{d}x =\int_{\frac{T}{2}}^T\mathscr{R}_0^4\mathrm{d}x
   +O\bigg(\int_{\frac{T}{2}}^T|\mathscr{R}_0|^3|\mathscr{R}_1+\mathscr{R}_2|\mathrm{d}x+\int_{\frac{T}{2}}^T|\mathscr{R}_1+\mathscr{R}_2|^4\mathrm{d}x\bigg).
\end{equation}
  According to the elementary formula
\begin{equation*}
\cos{a_1}\cos{a_2}\cdots\cos{a_k}=\frac{1}{2^{k-1}}\sum_{(i_1,i_2\cdots,i_{k-1})\in\{0,1\}^{k-1}}\cos\big(a_1+(-1)^{i_1}a_2+\cdots+(-1)^{i_{k-1}}a_k\big),
\end{equation*}
we have
\begin{align*}
   \mathscr{R}_0^4 = &\frac{x}{4\pi^4}\sum_{n_1,n_2,n_3,n_4\leqslant y}\frac{\tau(n_1,x)\tau(n_2,x)\tau(n_3,x)\tau(n_4,x)}{(n_1n_2n_3n_4)^{3/4}}
                               \nonumber \\
   = & \frac{x}{4\pi^4}\sum_{n_1,n_2,n_3,n_4\leqslant y}\frac{1}{(n_1n_2n_3n_4)^{3/4}}
                                \nonumber \\
   &   \qquad\times\sum_{\substack{n_j=h_jr_j\\j=1,2,3,4}}\prod_{j=1}^4
             \cos\bigg(4\pi\sqrt{n_jx}-2\pi\bigg(\frac{h_j\ell_2}{M_2}+\frac{r_j\ell_1}{M_1}+\frac{1}{8}\bigg)\bigg)
                               \nonumber \\
   = & \frac{x}{32\pi^4}\sum_{(i_1,i_2,i_3)\in\{0,1\}^3}\sum_{n_1,n_2,n_3,n_4\leqslant y}\frac{1}{(n_1n_2n_3n_4)^{3/4}}
                               \nonumber \\
   &   \qquad\times\sum_{\substack{n_j=h_jr_j\\j=1,2,3,4}}
        \cos\big(4\pi\boldsymbol{\alpha}(\mathbf{n},\mathbf{i})\sqrt{x}-2\pi\boldsymbol{\beta}(\mathbf{h},\mathbf{r},\mathbf{i})\big),
\end{align*}
where $\mathbf{n}=(n_1,n_2,n_3,n_4)\in\mathbb{N}^4,\,\mathbf{i}=(i_1,i_2,i_3)\in\{0,1\}^3$ and
\begin{align*}
            \boldsymbol{\alpha}(\mathbf{n},\mathbf{i})
   := & \sqrt{n_1}+(-1)^{i_1}\sqrt{n_2}+(-1)^{i_2}\sqrt{n_3}+(-1)^{i_3}\sqrt{n_4},  \\
          \boldsymbol{\beta}(\mathbf{h},\mathbf{r},\mathbf{i})
   := & \bigg(\frac{h_1\ell_2}{M_2}+\frac{r_1\ell_1}{M_1}+\frac{1}{8}\bigg)+(-1)^{i_1}\bigg(\frac{h_2\ell_2}{M_2}+\frac{r_2\ell_1}{M_1}+\frac{1}{8}\bigg)   \\
      & +(-1)^{i_2}\bigg(\frac{h_3\ell_2}{M_2}+\frac{r_3\ell_1}{M_1}+\frac{1}{8}\bigg)
        +(-1)^{i_3}\bigg(\frac{h_4\ell_2}{M_2}+\frac{r_4\ell_1}{M_1}+\frac{1}{8}\bigg).
\end{align*}
 Therefore, we can write
\begin{equation}\label{R_0=S_1(x)+S_2(x)}
  \mathscr{R}_0^4=S_1(x)+S_2(x),
\end{equation}
where
\begin{align*}
   S_1(x)= & \frac{x}{32\pi^4}\sum_{(i_1,i_2,i_3)\in\{0,1\}^3}\sum_{\substack{n_1,n_2,n_3,n_4\leqslant y\\ \boldsymbol{\alpha}(\mathbf{n},\mathbf{i})=0 }}
   \frac{1}{(n_1n_2n_3n_4)^{3/4}}
           \sum_{\substack{n_j=h_jr_j\\j=1,2,3,4}}\cos\big(2\pi\boldsymbol{\beta}(\mathbf{h},\mathbf{r},\mathbf{i})\big),  \\
   S_2(x)=& \frac{x}{32\pi^4}\sum_{(i_1,i_2,i_3)\in\{0,1\}^3}\sum_{\substack{n_1,n_2,n_3,n_4\leqslant y\\ \boldsymbol{\alpha}(\mathbf{n},\mathbf{i})\not=0 }}
            \frac{1}{(n_1n_2n_3n_4)^{3/4}}  \\
    & \qquad\qquad\qquad\qquad\qquad\qquad\times \sum_{\substack{n_j=h_jr_j\\j=1,2,3,4}}
        \cos\big(4\pi\boldsymbol{\alpha}(\mathbf{n},\mathbf{i})\sqrt{x}-2\pi\boldsymbol{\beta}(\mathbf{h},\mathbf{r},\mathbf{i})\big).
\end{align*}
For $\mathbf{n}=(n_1,n_2,n_3,n_4)\in\mathbb{N}^4,\,\mathbf{i}=(i_1,i_2,i_3)\in\{0,1\}^3$, define
\begin{equation*}
    \tau(\mathbf{n},\mathbf{i}):=\sum_{\substack{\mathbf{n}\in\mathbb{N}^4\\ \boldsymbol{\alpha}(\mathbf{n},\mathbf{i})=0}}
      \frac{1}{(n_1n_2n_3n_4)^{3/4}}
           \sum_{\substack{n_j=h_jr_j\\j=1,2,3,4}}\cos\big(2\pi\boldsymbol{\beta}(\mathbf{h},\mathbf{r},\mathbf{i})\big)
\end{equation*}
and
\begin{equation*}
    \tau(\mathbf{n},\mathbf{i};y):=\sum_{\substack{n_1,n_2,n_3,n_4\leqslant y\\ \boldsymbol{\alpha}(\mathbf{n},\mathbf{i})=0}}
      \frac{1}{(n_1n_2n_3n_4)^{3/4}}
           \sum_{\substack{n_j=h_jr_j\\j=1,2,3,4}}\cos\big(2\pi\boldsymbol{\beta}(\mathbf{h},\mathbf{r},\mathbf{i})\big).
\end{equation*}
The condition $\boldsymbol{\alpha}(\mathbf{n},\mathbf{i})=0$ implies $\mathbf{i}\not=\mathbf{0}$, i.e. $(i_1,i_2,i_3)\not=(0,0,0)$. It is easy to see
that for $\mathbf{i}=(i_1,i_2,i_3)\in\{0,1\}^3$ and $\mathbf{i}'=(i_1',i_2',i_3')\in\{0,1\}^3$ satisfying $i_1+i_2+i_3=i_1'+i_2'+i_3'=v$, there must hold
\begin{equation*}
  \tau(\mathbf{n},\mathbf{i})= \tau(\mathbf{n},\mathbf{i}')=:s_{4;v},\qquad \tau(\mathbf{n},\mathbf{i};y)= \tau(\mathbf{n},\mathbf{i}';y)=:s_{4;v}(y),
\end{equation*}
say. Also, we write
\begin{equation}\label{C_4(l_1/M_1,l_2/M_2)-def}
  C_4\bigg(\frac{\ell_1}{M_1},\frac{\ell_2}{M_2}\bigg)=\sum_{v=1}^3{3\choose v}s_{4;v}.
\end{equation}
 Thus, from Lemma \ref{s_k;v(f)=s_k;v(f;y)+error} we deduce that
\begin{align*}
  B_4\bigg(\frac{\ell_1}{M_1},\frac{\ell_2}{M_2};y\bigg):=\sum_{\mathbf{i}\in\{0,1\}^3}\tau(\mathbf{n},\mathbf{i};y)
  =\sum_{v=1}^3{3\choose v}s_{4;v}(y)=C_4\bigg(\frac{\ell_1}{M_1},\frac{\ell_2}{M_2}\bigg)+O(y^{-1/2+\varepsilon}).
\end{align*}
Hence, we obtain
\begin{align}\label{S_1(x)-asymptotic}
       \int_{\frac{T}{2}}^T S_1(x)\mathrm{d}x
   = & \frac{1}{32\pi^4}B_4\bigg(\frac{\ell_1}{M_1},\frac{\ell_2}{M_2};y\bigg)\int_{\frac{T}{2}}^T x\mathrm{d}x
             \nonumber \\
   = &  \frac{1}{32\pi^4} C_4\bigg(\frac{\ell_1}{M_1},\frac{\ell_2}{M_2}\bigg) \int_{\frac{T}{2}}^T x\mathrm{d}x + O(T^2y^{-1/2+\varepsilon})
            \nonumber \\
   = &  \frac{1}{32\pi^4} C_4\bigg(\frac{\ell_1}{M_1},\frac{\ell_2}{M_2}\bigg) \int_{\frac{T}{2}}^T x\mathrm{d}x + O(T^{2-1/8+\varepsilon}).
\end{align}
 Now, we proceed to consider the contribution of $S_2(x)$. By Lemma \ref{yijie}, we get
\begin{equation}\label{S_2(x)-upper-all}
  \int_{\frac{T}{2}}^TS_2(x)\mathrm{d}x\ll\sum_{\mathbf{i}\in\{0,1\}^3}
       \sum_{\substack{n_1,n_2,n_3,n_4\leqslant y\\ \boldsymbol{\alpha}(\mathbf{n},\mathbf{i})\not=0}}
       \frac{d(n_1)d(n_2)d(n_3)d(n_4)}{(n_1n_2n_3n_4)^{3/4}}\min\bigg(T^2,\frac{T^{3/2}}{|\boldsymbol{\alpha}(\mathbf{n},\mathbf{i})|}\bigg).
\end{equation}
 For $\mathbf{i}=(i_1,i_2,i_3)\in\{0,1\}^3$, if $(i_1,i_2,i_3)=(0,0,0)$, then the contribution is
\begin{align}\label{S_2(x)-all-positive}
   \ll & \,\, T^{3/2}\sum_{n_1,n_2,n_3,n_4\leqslant y}\frac{d(n_1)d(n_2)d(n_3)d(n_4)}{(n_1n_2n_3n_4)^{3/4}(\sqrt{n_1}+\sqrt{n_2}+\sqrt{n_3}+\sqrt{n_4})}
                   \nonumber  \\
   \ll & \,\, T^{3/2+\varepsilon}\sum_{n_1\leqslant n_2\leqslant n_3\leqslant n_4\leqslant y}\frac{1}{(n_1n_2n_3n_4)^{3/4} n_4^{1/2}}
                    \nonumber  \\
   \ll & \,\, T^{3/2+\varepsilon}\sum_{n_1\leqslant n_2\leqslant n_3\leqslant n_4\leqslant y}\frac{1}{(n_1n_2n_3)^{3/4}n_4^{5/4}}
                     \nonumber  \\
   \ll & \,\,  T^{3/2+\varepsilon}y^{1/2}\ll T^{2-1/8+\varepsilon}.
\end{align}
If $(i_1,i_2,i_3)\not=(0,0,0)$, we write
\begin{equation}\label{not-all>0-Sigma1+Sigma2}
  \sum_{\mathbf{0}\not=\mathbf{i}\in\{0,1\}^3}
       \sum_{\substack{n_1,n_2,n_3,n_4\leqslant y\\ \boldsymbol{\alpha}(\mathbf{n},\mathbf{i})\not=0}}
       \frac{d(n_1)d(n_2)d(n_3)d(n_4)}{(n_1n_2n_3n_4)^{3/4}}\min\bigg(T^2,\frac{T^{3/2}}{|\boldsymbol{\alpha}(\mathbf{n},\mathbf{i})|}\bigg)
       =\Sigma_{(1)}+\Sigma_{(2)},
\end{equation}
where
\begin{equation*}
 \Sigma_{(1)}= \sum_{\substack{\mathbf{i}\in\{0,1\}^3\\ i_1+i_2+i_3=2}}
       \sum_{\substack{n_1,n_2,n_3,n_4\leqslant y\\ \boldsymbol{\alpha}(\mathbf{n},\mathbf{i})\not=0}}
       \frac{d(n_1)d(n_2)d(n_3)d(n_4)}{(n_1n_2n_3n_4)^{3/4}}\min\bigg(T^2,\frac{T^{3/2}}{|\boldsymbol{\alpha}(\mathbf{n},\mathbf{i})|}\bigg),
\end{equation*}
\begin{equation*}
 \Sigma_{(2)}= \sum_{\substack{\mathbf{i}\in\{0,1\}^3\\ i_1+i_2+i_3=1}}
       \sum_{\substack{n_1,n_2,n_3,n_4\leqslant y\\ \boldsymbol{\alpha}(\mathbf{n},\mathbf{i})\not=0}}
       \frac{d(n_1)d(n_2)d(n_3)d(n_4)}{(n_1n_2n_3n_4)^{3/4}}\min\bigg(T^2,\frac{T^{3/2}}{|\boldsymbol{\alpha}(\mathbf{n},\mathbf{i})|}\bigg).
\end{equation*}
Let
\begin{equation*}
   g=g(n_1,n_2,n_3,n_4):=\frac{d(n_1)d(n_2)d(n_3)d(n_4)}{(n_1n_2n_3n_4)^{3/4}}.
\end{equation*}
For $\Sigma_{(1)}$, by a splitting argument, there exist $1\leqslant N_1,N_2,N_3,N_4\leqslant y$ such that
\begin{equation}\label{Sigma_(1)-fenjie}
   \Sigma_{(1)}\ll T^\varepsilon \mathcal{G}(N_1,N_2,N_3,N_4)
\end{equation}
where
\begin{align*}
       \mathcal{G}(N_1,N_2,N_3,N_4)
   = & \sum_{\substack{\sqrt{n_1}+\sqrt{n_2}\not=\sqrt{n_3}+\sqrt{n_4}\\ n_i\sim N_i,1\leqslant i\leqslant 4\\1\leqslant N_1\leqslant N_2\leqslant y\\
           1\leqslant N_3\leqslant N_4\leqslant y}}g \cdot
       \min\bigg(T^2,\frac{T^{3/2}}{|\sqrt{n_1}+\sqrt{n_2}-\sqrt{n_3}-\sqrt{n_4}|}\bigg)
\end{align*}
If $N_2\geqslant300N_4$, then $|\sqrt{n_1}+\sqrt{n_2}-\sqrt{n_3}-\sqrt{n_4}|\gg N_2^{1/2}$, so the trivial estimate yields
\begin{equation*}
  \mathcal{G}(N_1,N_2,N_3,N_4)\ll \frac{T^{3/2+\varepsilon}N_1N_2N_3N_4}{(N_1N_2N_3N_4)^{3/4}N_2^{1/2}}\ll
      T^{3/2+\varepsilon}y^{1/2}\ll T^{2-1/8+\varepsilon}.
\end{equation*}
If $N_4\geqslant300N_2$, we can get the same estimate.  So later we always suppose that $N_2\asymp N_4$. Let $\eta=\sqrt{n_1}+\sqrt{n_2}-\sqrt{n_3}-\sqrt{n_4}$.
Write
\begin{equation}\label{G-fenjie}
 \mathcal{G}(N_1,N_2,N_3,N_4)=\mathcal{G}_1+\mathcal{G}_2+\mathcal{G}_3,
\end{equation}
where
\begin{eqnarray*}
   \mathcal{G}_1 & := & T^{2}\sum_{\substack{0<|\eta|\leqslant T^{-1/2}\\n_i\sim N_i,1\leqslant i\leqslant4
                                      \\1\leqslant N_1\leqslant N_2\leqslant y\\ 1\leqslant N_3\leqslant N_4\leqslant y}}g, \\
   \mathcal{G}_2 & := & T^{3/2}\sum_{\substack{ T^{-1/2}<|\eta|\leqslant 1 \\n_i\sim N_i,1\leqslant i\leqslant4
                                      \\1\leqslant N_1\leqslant N_2\leqslant y\\ 1\leqslant N_3\leqslant N_4\leqslant y   }}g|\eta|^{-1}, \\
\end{eqnarray*}
\begin{eqnarray*}
   \mathcal{G}_3 & := & T^{3/2}\sum_{\substack{|\eta|>1 \\n_i\sim N_i,1\leqslant i\leqslant4
                                      \\1\leqslant N_1\leqslant N_2\leqslant y\\ 1\leqslant N_3\leqslant N_4\leqslant y}}g|\eta|^{-1}.
\end{eqnarray*}
We estimate $\mathcal{G}_1$ first. By Lemma \ref{Zhai-lemma-5}, we get
\begin{eqnarray}\label{G_1-111}
   \mathcal{G}_1 & \ll & \frac{T^{2+\varepsilon}}{(N_1N_2N_3N_4)^{3/4}}\mathscr{A}_1\big(N_1,N_2,N_3,N_4;T^{-1/2}\big) \nonumber  \\
     & \ll & \frac{T^{2+\varepsilon}}{(N_1N_2N_3N_4)^{3/4}}\big(T^{-1/2}N_4^{1/2}N_1N_2N_3+N_1N_3N_4^{1/2}\big)    \nonumber  \\
     & \ll & T^{3/2+\varepsilon}(N_1N_3)^{1/4}+T^{2+\varepsilon}(N_1N_3)^{1/4}N_4^{-1}     \nonumber  \\
     & \ll & T^{3/2+\varepsilon}y^{1/2}+T^{2+\varepsilon}(N_1N_3)^{1/4}N_4^{-1}        \nonumber  \\
     & \ll & T^{2-1/8+\varepsilon}+T^{2+\varepsilon}(N_1N_3)^{1/4}N_4^{-1}.
\end{eqnarray}
On the other hand, by Lemma \ref{Zhai-lemma-3}, without loss of generality, we assume that $N_1\leqslant N_3\leqslant N_4$ and obtain
\begin{eqnarray}\label{G-1-222}
  \mathcal{G}_1 & \ll & \frac{T^{2+\varepsilon}}{(N_1N_2N_3N_4)^{3/4}}\mathscr{A}_{-}(N_1,N_2,N_3,N_4;T^{-1/2})   \nonumber  \\
  & \ll & \frac{T^{2+\varepsilon}}{(N_1N_2N_3N_4)^{3/4}} \big(T^{-1/8}N_1^{7/8}+N_1^{1/2}\big)
          \big(T^{-1/8}N_3^{7/8}+N_3^{1/2}\big)\big(T^{-1/4}N_4^{7/4}+N_4\big)      \nonumber  \\
  & \ll & T^{2+\varepsilon}(N_1N_3)^{-1/4}N_4^{-1/2}\big(T^{-1/8}N_1^{3/8}+1\big)\big(T^{-1/8}N_3^{3/8}+1\big)\big(T^{-1/4}N_4^{3/4}+1\big) \nonumber  \\
  & \ll & T^{2+\varepsilon}(N_1N_3)^{-1/4}N_4^{-1/2}
          \big(T^{-1/4}(N_1N_3)^{3/8}+T^{-1/8}N_3^{3/8}+1\big) \big(T^{-1/4}N_4^{3/4}+1\big)    \nonumber  \\
  & \ll & T^{7/4+\varepsilon}(N_1N_3)^{1/8}N_4^{-1/2}    \nonumber  \\
  &     & +T^{2+\varepsilon}(N_1N_3)^{-1/4}N_4^{-1/2} \big(T^{-1/8}N_3^{3/8}+1\big) \big(T^{-1/4}N_4^{3/4}+1\big)  \nonumber  \\
  & \ll & T^{7/4+\varepsilon}N_4^{-1/4}+T^{2+\varepsilon}(N_1N_3)^{-1/4}N_4^{-1/2}\big(T^{-3/8}N_4^{9/8}+1\big)    \nonumber  \\
  & \ll & T^{7/4+\varepsilon}+T^{2+\varepsilon}(N_1K_3)^{-1/4}N_4^{-1/2}\big(T^{-3/8}N_4^{9/8}+1\big).
\end{eqnarray}
From (\ref{G_1-111}) and (\ref{G-1-222}), we get
\begin{equation*}
   \mathcal{G}_1  \ll  T^{2-1/8+\varepsilon}+T^{2+\varepsilon}\cdot
   \min\Bigg(\frac{(N_1N_3)^{1/4}}{N_4},\frac{T^{-3/8}N_4^{9/8}+1}{(N_1N_3)^{1/4}N_4^{1/2}}\Bigg).
\end{equation*}
\textbf{Case 1} If $N_4\gg T^{1/3}$, then $T^{-3/8}N_4^{9/8}\gg1,$ and thus
\begin{eqnarray}\label{G_1-case-1}
  \mathcal{G}_1 & \ll & T^{2-1/8+\varepsilon}+T^{2+\varepsilon}\cdot
   \min\Bigg(\frac{(N_1N_3)^{1/4}}{N_4},\frac{T^{-3/8}N_4^{9/8}}{(N_1N_3)^{1/4}N_4^{1/2}}\bigg) \nonumber \\
                         & \ll & T^{2-1/8+\varepsilon}+T^{2+\varepsilon}\bigg(\frac{(N_1N_3)^{1/4}}{N_4}\bigg)^{1/2}
          \bigg(\frac{T^{-3/8}N_4^{9/8}}{(N_1N_3)^{1/4}N_4^{1/2}}\bigg)^{1/2}   \nonumber \\
    & \ll & T^{2-1/8+\varepsilon}+T^{2-3/16+\varepsilon}N_4^{-3/16}\ll T^{2-1/8+\varepsilon}.
\end{eqnarray}
\textbf{Case 2} If $N_4\ll T^{1/3}$, then $T^{-3/8}N_4^{9/8}\ll1$.
By noting  Lemma \ref{kong-lamma} and
\begin{equation*}
N_2\asymp N_4\asymp\max(N_1,N_2,N_3,N_4),
\end{equation*}
we have
\begin{equation*}
 T^{-1/2}\gg|\eta|\gg(n_1n_2n_3n_4)^{-1/2}\max(n_1,n_2,n_3,n_4)^{-3/2}\asymp (N_1N_3)^{-1/2}N_4^{-5/2}.
\end{equation*}
Hence, we obtain
\begin{eqnarray}\label{G_1-case-2}
   \mathcal{G}_1 & \ll & T^{2-1/8+\varepsilon}+T^{2+\varepsilon}\min\Bigg(\frac{(N_1N_3)^{1/4}}{N_4},\frac{1}{(N_1N_3)^{1/4}N_4^{1/2}}\Bigg)
                           \nonumber \\
   & \ll & T^{2-1/8+\varepsilon}+T^{2+\varepsilon}\Bigg(\frac{(N_1N_3)^{1/4}}{N_4}\Bigg)^{1/4}\Bigg(\frac{1}{(N_1N_3)^{1/4}N_4^{1/2}}\Bigg)^{3/4}
                            \nonumber \\
   & = & T^{2-1/8+\varepsilon}+T^{2+\varepsilon}(N_1N_3)^{-1/8}N_4^{-5/8}
                            \nonumber \\
   & \ll & T^{2-1/8+\varepsilon}+T^{2+\varepsilon}(T^{-1/2})^{1/4}\ll T^{2-1/8+\varepsilon}.
\end{eqnarray}
Combining (\ref{G_1-case-1}) and (\ref{G_1-case-2}), we get
\begin{equation}\label{G-1-estimate}
   \mathcal{G}_1\ll T^{2-1/8+\varepsilon}.
\end{equation}

 Now, we estimate $\mathcal{G}_2$. By a splitting argument, we get that there exists some $\delta$ satisfying $T^{-1/2}\ll \delta\ll 1$ such that
\begin{equation*}
  \mathcal{G}_2\ll\frac{T^{3/2+\varepsilon}}{(N_1N_2N_3N_4)^{3/4}\delta}\times\sum_{\substack{\delta<|\eta|\leqslant2\delta\\ \eta\not=0}}1.
\end{equation*}
By Lemma \ref{Zhai-lemma-5}, we get
\begin{eqnarray}\label{G-2-111}
  \mathcal{G}_2 & \ll & \frac{T^{3/2+\varepsilon}}{(N_1N_2N_3N_4)^{3/4}\delta}\mathscr{A}_1(N_1,N_2,N_3,N_4;2\delta)
                               \nonumber \\
     & \ll & \frac{T^{3/2+\varepsilon}}{(N_1N_2N_3N_4)^{3/4}\delta} (\delta N_4^{1/2}N_1N_2N_3+N_1N_3N_4^{1/2})
                                \nonumber \\
     & = &   T^{3/2+\varepsilon}(N_1N_3)^{1/4}+T^{3/2+\varepsilon}\delta^{-1}(N_1N_3)^{1/4}N_4^{-1}
                                \nonumber \\
     & \ll &  T^{3/2+\varepsilon}y^{1/2}+T^{3/2+\varepsilon}\delta^{-1}(N_1N_3)^{1/4}N_4^{-1}
                                \nonumber \\
     & \ll &  T^{2-1/8+\varepsilon}+T^{3/2+\varepsilon}\delta^{-1}(N_1N_3)^{1/4}N_4^{-1} .
\end{eqnarray}
On the other hand, by Lemma \ref{Zhai-lemma-3}, without loss of generality, we assume that $N_1\leqslant N_3\leqslant N_4$ and obtain
\begin{eqnarray*}
  \mathcal{G}_2 & \ll & \frac{T^{3/2+\varepsilon}}{(N_1N_2N_3N_4)^{3/4}\delta}\times\mathscr{A}_{-}(N_1,N_2,N_3,N_4;2\delta)
                           \nonumber   \\
   & \ll & \frac{T^{3/2+\varepsilon}}{(N_1N_2N_3N_4)^{3/4}\delta} \big(\delta^{1/4}N_1^{7/8}+N_1^{1/2}\big)
           \big(\delta^{1/4}N_3^{7/8}+N_3^{1/2}\big)\big(\delta^{1/2}N_4^{7/4}+N_4\big)
                            \nonumber   \\
 & \ll & T^{3/2+\varepsilon}(N_1N_3)^{-1/4}N_4^{-1/2}\delta^{-1}\big(\delta^{1/4}N_1^{3/8}+1\big)
           \big(\delta^{1/4}N_3^{3/8}+1\big)\big(\delta^{1/2}N_4^{3/4}+1\big)
\end{eqnarray*}
\begin{eqnarray}\label{G-2-222}
   & \ll &  T^{3/2+\varepsilon}(N_1N_3)^{-1/4}N_4^{-1/2}\delta^{-1}
            \big(\delta^{1/2}(N_1N_3)^{3/8}+\delta^{1/4}N_3^{3/8}+1\big) \big(\delta^{1/2}N_4^{3/4}+1\big)
                           \nonumber   \\
   & \ll & T^{3/2+\varepsilon}(N_1N_3)^{1/8}N_4^{-1/2}\delta^{-1/2}
                           \nonumber   \\
   &     &  +T^{3/2+\varepsilon}(N_1N_3)^{-1/4}N_4^{-1/2}\delta^{-1}\big(\delta^{1/4}N_3^{3/8}+1\big) \big(\delta^{1/2}N_4^{3/4}+1\big)
                            \nonumber   \\
   & \ll & T^{7/4+\varepsilon}N_4^{-1/4}+T^{3/2+\varepsilon}(N_1N_3)^{-1/4}N_4^{-1/2}\delta^{-1}\big(\delta^{3/4}N_4^{9/8}+1\big).
\end{eqnarray}
From (\ref{G-2-111}) and (\ref{G-2-222}), we get
\begin{equation*}
 \mathcal{G}_2  \ll  T^{2-1/8+\varepsilon}+T^{3/2+\varepsilon}\delta^{-1}
                         \cdot\min\Bigg(\frac{(N_1N_3)^{1/4}}{N_4},\frac{\delta^{3/4}N_4^{9/8}+1}{(N_1N_3)^{1/4}N_4^{1/2}}\Bigg).
\end{equation*}
\textbf{Case 1} If $\delta\gg N_4^{-3/2}$, then $\delta^{3/4}N_4^{9/8}\gg1$, we get (recall $\delta\gg T^{-1/2}$)
\begin{eqnarray}\label{G_2-case-1}
   \mathcal{G}_2 & \ll & T^{2-1/8+\varepsilon}+T^{3/2+\varepsilon}\delta^{-1}
                         \cdot\min\Bigg(\frac{(N_1N_3)^{1/4}}{N_4},\frac{\delta^{3/4}N_4^{9/8}}{(N_1N_3)^{1/4}N_4^{1/2}}\Bigg)
                          \nonumber   \\
              & \ll &  T^{2-1/8+\varepsilon}+ T^{3/2+\varepsilon}\delta^{-1}
                      \Bigg(\frac{(N_1N_3)^{1/4}}{N_4}\Bigg)^{1/2}   \Bigg(\frac{\delta^{3/4}N_4^{9/8}}{(N_1N_3)^{1/4}N_4^{1/2}}\Bigg)^{1/2}
                         \nonumber   \\
              & \ll & T^{2-1/8+\varepsilon} +T^{3/2+\varepsilon}\delta^{-5/8}N_4^{-3/16}
                         \nonumber   \\
              & \ll & T^{2-1/8+\varepsilon}+T^{3/2+\varepsilon}T^{5/16}N_4^{-3/16}\ll T^{2-1/8+\varepsilon}.
\end{eqnarray}
\textbf{Case 2} If $\delta\ll N_4^{-3/2}$, then $\delta^{3/4}N_4^{9/8}\ll1$. By Lemma \ref{kong-lamma}, we have
\begin{equation*}
   \delta\gg|\eta|\gg(n_1n_2n_3n_4)^{-1/2}\max(n_1,n_2,n_3,n_4)^{-3/2}\asymp (N_1N_3)^{-1/2}N_4^{-5/2}.
\end{equation*}
 Therefore, we obtain (recall $\delta\gg T^{-1/2}$)
\begin{eqnarray} \label{G_2-case-2}
  \mathcal{G}_2 & \ll & T^{2-1/8+\varepsilon}+T^{3/2+\varepsilon}\delta^{-1}\cdot
                         \min\Bigg(\frac{(N_1N_3)^{1/4}}{N_4},\frac{1}{(N_1N_3)^{1/4}N_4^{1/2}}\Bigg)
                              \nonumber   \\
    & \ll & T^{2-1/8+\varepsilon}+T^{3/2+\varepsilon}\delta^{-1}\Bigg(\frac{(N_1N_3)^{1/4}}{N_4}\Bigg)^{1/4}
             \Bigg(\frac{1}{(N_1N_3)^{1/4}N_4^{1/2}}\Bigg)^{3/4}
                              \nonumber   \\
    & \ll & T^{2-1/8+\varepsilon}+T^{3/2+\varepsilon}\delta^{-1}(N_1N_3)^{-1/8}N_4^{-5/8}
                             \nonumber   \\
     & \ll & T^{2-1/8+\varepsilon}+T^{3/2+\varepsilon}\delta^{-1}\delta^{1/4}
                               \nonumber   \\
     & \ll & T^{2-1/8+\varepsilon}+T^{3/2+\varepsilon}T^{3/8}\ll  T^{2-1/8+\varepsilon}.
\end{eqnarray}
Combining (\ref{G_2-case-1}) and (\ref{G_2-case-2}), we get
\begin{equation}\label{G-2-estimate}
   \mathcal{G}_2\ll T^{2-1/8+\varepsilon}.
\end{equation}
For $\mathcal{G}_3$, by a splitting argument and Lemma \ref{Zhai-lemma-5} again, we get
\begin{eqnarray*}
   \mathcal{G}_3 & \ll & \frac{T^{3/2+\varepsilon}}{(N_1N_2N_3N_4)^{3/4}\delta}\times
                         \sum_{\substack{\delta<|\eta|\leqslant2\delta\\ \delta\gg 1}}1
\end{eqnarray*}
\begin{eqnarray} \label{G-3-estimate}
    & \ll & \frac{T^{3/2+\varepsilon}}{(N_1N_2N_3N_4)^{3/4}\delta}\cdot \delta N_4^{1/2}N_1N_2N_3\ll  T^{3/2+\varepsilon}(N_1N_3)^{1/4}  \nonumber \\
    & \ll &  T^{3/2+\varepsilon}y^{1/2}\ll T^{2-1/8+\varepsilon}.
\end{eqnarray}
Comnining (\ref{Sigma_(1)-fenjie}), (\ref{G-fenjie}), (\ref{G-1-estimate}), (\ref{G-2-estimate}) and (\ref{G-3-estimate}), we deduce that
\begin{equation}\label{Sigma_{(1)}-upper}
   \Sigma_{(1)}\ll T^{2-1/8+\varepsilon}.
\end{equation}
In the same way, by using Lemma \ref{Zhai-lemma-6} instead of Lemma \ref{Zhai-lemma-5}, we can follow the above arguments step by step and obtain
\begin{equation}\label{Sigma_{(2)}-upper}
   \Sigma_{(2)}\ll T^{2-1/8+\varepsilon}.
\end{equation}
It follows from (\ref{S_2(x)-upper-all}), (\ref{S_2(x)-all-positive}), (\ref{not-all>0-Sigma1+Sigma2}),
(\ref{Sigma_{(1)}-upper}) and (\ref{Sigma_{(2)}-upper}) that
\begin{equation}\label{S_2(x)-upper-last}
   \int_\frac{T}{2}^TS_2(x)\mathrm{d}x\ll T^{2-1/8+\varepsilon}.
\end{equation}
According to (\ref{R_0=S_1(x)+S_2(x)}), (\ref{S_1(x)-asymptotic}) and (\ref{S_2(x)-upper-last}), we conclude that
\begin{equation}\label{R_0-fourth-power-mean}
  \int_{\frac{T}{2}}^T\mathscr{R}_0^4\mathrm{d}x=\frac{1}{32\pi^4}C_4\bigg(\frac{\ell_1}{M_1},\frac{\ell_2}{M_2}\bigg)
  \int_{\frac{T}{2}}^Tx\mathrm{d}x+O(T^{2-1/8+\varepsilon}).
\end{equation}
From (4.11) of Liu \cite{Liu-Kui}, we know that
\begin{equation}\label{R_1+R_2-mean-square}
  \int_{\frac{T}{2}}^T|\mathscr{R}_1+\mathscr{R}_2|^2\mathrm{d}x\ll T^{1/2}y^{-1/2}\mathscr{L}^3.
\end{equation}
Also, by (4.14) of Liu \cite{Liu-Kui}, we have
\begin{equation}\label{R_1+R_2-mean-8th-power}
  \int_{\frac{T}{2}}^T|\mathscr{R}_1+\mathscr{R}_2|^8\mathrm{d}x\ll T^{3+\varepsilon}.
\end{equation}
It follows from (\ref{R_1+R_2-mean-square}) and (\ref{R_1+R_2-mean-8th-power}) and Cauchy's inequality that
\begin{align}\label{R_1+R_2-mean-4th-power}
   \int_{\frac{T}{2}}^T|\mathscr{R}_1+\mathscr{R}_2|^4\mathrm{d}x
   = &  \int_{\frac{T}{2}}^T|\mathscr{R}_1+\mathscr{R}_2|^{4/3}|\mathscr{R}_1+\mathscr{R}_2|^{8/3}\mathrm{d}x
                       \nonumber \\
   \ll &  \bigg(\int_{\frac{T}{2}}^T|\mathscr{R}_1+\mathscr{R}_2|^2\mathrm{d}x\bigg)^{2/3}
          \bigg(\int_{\frac{T}{2}}^T|\mathscr{R}_1+\mathscr{R}_2|^8\mathrm{d}x\bigg)^{1/3}
                        \nonumber \\
   \ll & (T^{1/2}y^{-1/2}\mathscr{L}^3)^{2/3}(T^{3+\varepsilon})^{1/3}
                        \nonumber \\
   \ll &  T^{4/3+\varepsilon}y^{-\frac{1}{3}}\ll T^{13/12+\varepsilon}.
\end{align}
By (\ref{R_0-fourth-power-mean}), (\ref{R_1+R_2-mean-4th-power}) and H\"{o}lder's inequality, we deduce that
\begin{align}\label{R_^3|R_1+R_2|-mean-value}
  \int_{\frac{T}{2}}^T|\mathscr{R}_0|^3|\mathscr{R}_1+\mathscr{R}_2|\mathrm{d}x
  \ll & \bigg(\int_{\frac{T}{2}}^T\mathscr{R}_0^4\mathrm{d}x\bigg)^{3/4}\bigg(\int_{\frac{T}{2}}^T|\mathscr{R}_1+\mathscr{R}_2|^4\mathrm{d}x\bigg)^{1/4}
                    \nonumber \\
  \ll & (T^2)^{3/4}(T^{13/12+\varepsilon})^{1/4}\ll T^{85/48+\varepsilon}\ll T^{2-1/8+\varepsilon}.
\end{align}
Combining (\ref{R_4-mean-explicit}), (\ref{R_0-fourth-power-mean}), (\ref{R_1+R_2-mean-4th-power}) and (\ref{R_^3|R_1+R_2|-mean-value}), we obtain
\begin{equation}\label{R^4-asymp-conclusion}
  \int_{\frac{T}{2}}^T\mathscr{R}^4\mathrm{d}x=\frac{1}{32\pi^4}C_4\bigg(\frac{\ell_1}{M_1},\frac{\ell_2}{M_2}\bigg)
  \int_{\frac{T}{2}}^Tx\mathrm{d}x+O(T^{2-1/8+\varepsilon}).
\end{equation}
This completes the proof of Proposition \ref{proposition-1}.
\end{proof}

\begin{proposition}\label{proposition-2}
   For $H\gg T^\varepsilon$, we have
\begin{equation*}
\int_{\frac{T}{2}}^T\mathscr{G}^4\mathrm{d}x\ll T^{3/2}H^{-1/4},
\end{equation*}
\begin{equation*}
\int_{\frac{T}{2}}^T|\mathscr{R}|^3|\mathscr{G}|\mathrm{d}x\ll T^{2-1/8+\varepsilon},
\end{equation*}
\begin{equation*}
\int_{\frac{T}{2}}^T|\mathscr{R}|^3\mathscr{L}^3\mathrm{d}x\ll T^{2-1/8+\varepsilon}.
\end{equation*}
\end{proposition}
\begin{proof}
 Recall that
\begin{equation*}
\mathscr{G}=G_{12}(x;H)+G_{21}(x;H).
\end{equation*}
From (\ref{G_12(x;H)-def}), we deduce that
\begin{align}\label{G_12(x;H)-fourth-power}
          \int_{\frac{T}{2}}^T\big(G_{12}(x;H)\big)^4\mathrm{d}x
  \ll & \sum_{\substack{n_1\leqslant\sqrt{M_1M_2T}\\ n_1\equiv \ell_1(\!\bmod M_1)}}
        \int_{\frac{T}{2}}^T\min\Bigg(1,\frac{1}{H\big\|\frac{M_1x}{n_1}-\frac{\ell_2}{M_2}\big\|^4}\Bigg)\mathrm{d}x
                 \nonumber \\
  \ll & \sum_{\substack{n_1\leqslant\sqrt{M_1M_2T}\\ n_1\equiv \ell_1(\!\bmod M_1)}}\frac{n_1}{M_1}
  \int_{\frac{M_1T}{2n_1}-\frac{\ell_2}{M_2}}^{\frac{M_1T}{n_1}-\frac{\ell_2}{M_2}}\min\bigg(1,\frac{1}{H\|u\|^4}\bigg)\mathrm{d}u
                  \nonumber \\
  \ll & \,\,T\sum_{\substack{n_1\leqslant\sqrt{M_1M_2T}\\ n_1\equiv \ell_1(\!\bmod M_1)}}
        \int_0^1\min\bigg(1,\frac{1}{H\|U\|^4}\bigg)\mathrm{d}u
                  \nonumber \\
   \ll & \,\,T\sum_{\substack{n_1\leqslant\sqrt{M_1M_2T}\\ n_1\equiv \ell_1(\!\bmod M_1)}}
         \bigg\{\int_0^{\frac{1}{2}}+\int_{\frac{1}{2}}^1\bigg\}\min\bigg(1,\frac{1}{H\|U\|^4}\bigg)\mathrm{d}u
                   \nonumber \\
   \ll & \,\, T\sum_{\substack{n_1\leqslant\sqrt{M_1M_2T}\\ n_1\equiv \ell_1(\!\bmod M_1)}}
               \int_0^{\frac{1}{2}}\min\bigg(1,\frac{1}{Hu^4}\bigg)\mathrm{d}u
                    \nonumber \\
   \ll & \,\, T\sum_{\substack{n_1\leqslant\sqrt{M_1M_2T}\\ n_1\equiv \ell_1(\!\bmod M_1)}}
              \bigg(\int_0^{H^{-\frac{1}{4}}}1\mathrm{d}u+\int_{H^{-\frac{1}{4}}}^{\frac{1}{2}}\frac{1}{Hu^4}\mathrm{d}u\bigg)
                    \nonumber \\
   \ll & \,\, T^{3/2}H^{-1/4}.
\end{align}
 Similarly, we also have
\begin{equation}\label{G_21(x;H)-fourth-power}
     \int_{\frac{T}{2}}^T\big(G_{21}(x;H)\big)^4\mathrm{d}x   \ll T^{3/2}H^{-1/4}.
\end{equation}
Therefore, from (\ref{G_12(x;H)-fourth-power}) and (\ref{G_21(x;H)-fourth-power}), we derive that
\begin{equation}\label{G-fourth-power}
     \int_{\frac{T}{2}}^T\mathscr{G}^4\mathrm{d}x   \ll T^{3/2}H^{-1/4}.
\end{equation}
 It follows from (\ref{R^4-asymp-conclusion}), (\ref{G-fourth-power}) and Holder's inequality that
\begin{align}
        \int_{\frac{T}{2}}^T|\mathscr{R}|^3|\mathscr{G}|\mathrm{d}x
  \ll & \,\,\bigg( \int_{\frac{T}{2}}^T|\mathscr{R}|^4\mathrm{d}x\bigg)^{3/4}\bigg(\int_{\frac{T}{2}}^T|\mathscr{G}|^4\mathrm{d}x\bigg)^{1/4}
                \nonumber \\
  \ll &  \,\,T^{15/8}H^{-1/16}\ll T^{2-1/8+\varepsilon},
\end{align}
and
\begin{equation}
   \int_{\frac{T}{2}}^T|\mathscr{R}|^3\mathscr{L}^3\mathrm{d}x
    \ll \bigg( \int_{\frac{T}{2}}^T\mathscr{R}^4\mathrm{d}x\bigg)^{3/4}\bigg(\int_{\frac{T}{2}}^T \mathscr{L}^{12}\mathrm{d}x\bigg)^{1/4}
    \ll T^{7/4}\mathscr{L}^3\ll T^{2-1/8+\varepsilon}.
\end{equation}
This completes the proof of Proposition \ref{proposition-2}.
\end{proof}

From  (\ref{Delta(M_1M_2x)-fenjie}), Proposition \ref{proposition-1} and  Proposition \ref{proposition-2}, we get the fourth--power moment of
$\Delta(M_1M_2x;\ell_1,M_1,\ell_2,M_2)$.

\section*{Acknowledgement}

   The authors would like to express the most sincere gratitude to Professor Wenguang Zhai for his valuable
advice and constant encouragement.

\end{document}